\documentclass[a4paper,11pt,english,oneside]{amsart}

\usepackage{arXiv}

\bmdefine{\bD}{D} 

\makeatletter
    \newcommand{\spectraldensalt}{\@ifstar{\spectraldensalt@Star}{\spectraldensalt@NoStar}}
        \newcommand{\spectraldensalt@NoStar}{\spectraldenssup{\ast\mspace{-4mu}}}
        \newcommand{\spectraldensalt@Star}[2][]{\spectraldenssup*[#1]{\ast\mspace{-4mu}}{#2}}
    \@DeclareFunction*[mathit]{spectralfunalt}{g}
\makeatother

\begin{document}
\title[$D$-norms, generators, and spectral functions]{A characterization of $\bD$-norms and their generators based on the family of spectral functions}
\author{Stefan Aulbach}
\address{University of W\"{u}rzburg\\ Institute of Mathematics\\ Emil-Fischer-Str.~30\\ 97074~W\"{u}rz\-burg, Germany}
\email{stefan.aulbach@uni-wuerzburg.de}
\thanks{The author was supported by DFG grant FA 262/4-1.}

\begin{abstract}
\Citet{aulfaho11} introduced the concept of $D$-norms in the framework of
functional extreme value theory (EVT) extending the multivariate case in a
natural manner. In particular, the distribution of a standard max-stable
process (MSP) $\bfeta\in \Cspace$ is completely determined by its
functional distribution function, which itself is given by some $D$-norm.

In order to generate a generalized Pareto process (GPP) that is in the
functional domain of attraction of $\bfeta$, one may use the fact that
every $D$-norm is defined by some generator process with continuous sample
paths. It is, however, still unknown which generator must be chosen such
that a \emph{given} $D$-norm arises. This is the content of the present
paper. We will, moreover, show that a generator process may be decomposed
into a functional deterministic part and a univariate random one.
\end{abstract}

\maketitle 

\section{Introduction}

\Citet{aulfaho11} introduced the concept of $D$-norms in the framework of
functional extreme value theory (EVT) extending the multivariate case in a
natural manner. In particular, the functional distribution function of any
standard max-stable process (MSP) $\bfeta = \opro{\eta}{t}{\unit}$ in
$\Cspace$, the set of all continuous functions on $\unit$, can be written in
terms of its $D$-norm
\begin{equation*}
  \P{\bfeta\le \f} = \exp*{-\dnorm{\f}}, \qquad \f\in \Ebarmin,
\end{equation*}
where $\Ebarmin := \set*{\f\in \Espace}{\f\le 0}$ with $\Espace$ being the
set of all those real valued functions on $\unit$ that are bounded and have
a finite set of discontinuities. For convenience we write stochastic
processes in bold font, such as $\bfeta$, whereas deterministic functions
are written in normal font; each operator and relation applied to
(stochastic and deterministic) functions is to be read pointwise.

Moreover, each standard generalized Pareto process (GPP) $\bfV =
\opro{V}{t}{\unit}\in \Cspace$ in the functional domain of attraction of
$\bfeta$ has the property
\begin{equation*}
  \P{\bfV\le \f} = 1 - \dnorm{\f}, \qquad \f\in \Ebarmin,\ \supnorm{\f}\le x_0,
\end{equation*}
for some $x_0>0$. The previous equations immediately imply their well-known
multivariate counterparts; simply choose $\f=\sum_{i\le d} x_i
\ind{\set{t_i}}$ for arbitrary $d\in\N$, $0\le t_1<\dots<t_d\le 1$, and
$x_1,\dots,x_d< 0$. We also refer to \citet{buihz08}, \citet{fahure10}, and
\citet{ferrdh12}.

According to \citet{aulfaho11}, a $D$-norm $\dnormd$ is defined by a
\emph{generator} process $\bfZ = \opro{Z}{t}{\unit}$ in $\Cbarplus :=
\set*{\h\in\Cspace}{\h\ge 0}$, namely
\begin{equation} \label{eqn:generator_D-norm}
  \dnorm{\f} = \E{\supnorm{\f\bfZ}}, \qquad \f\in \Espace,
\end{equation}
where $\f\bfZ := \opro*{\f*{t}Z_t}{t}{\unit}\in\Ebarmin$. Recall that a
generator is characterized by $\E{\bfZ}=1$, $t\in\unit$, and
$\E{\supnorm{Z_t}} < \infty$. Furthermore, $\dnormd$ defines a norm on
$\Espace$, \ie a generator $\bfZ$ maps the supremum norm $\supnormd$ onto
the $D$-norm $\dnormd = \E{\supnorm{{\cdot}\bfZ}}$. The case $\bfZ\equiv1$ shows
in particular that $\supnormd$ itself is a $D$-norm, too.

Changing the point of view we get that \eqref{eqn:generator_D-norm} defines
a function on the set of all generators onto the set of all $D$-norms.
Therefore, we call two generators \emph{equivalent} if they are mapped on
the same $D$-norm. For example, any generator $\bfZ\equiv Z$, $Z$ being any
non-negative random variable satisfying $\E{Z}=1$, is equivalent with the
constant function $1$.

It is, however, still unknown which generator $\bfZ$ must be chosen such
that a
\emph{given} $D$-norm 
arises. This is the content of the present paper which is organized as
follows: In \prettyref{sec:generators_special_D-norms} we construct the
generator process for a certain subclass of $D$-norms. These result from a
simple one parametric model that was already considered in \citet{dehap06},
\citet{falk11}, and \citet{aulfa11}. \prettyref{sec:generator_decomposition}
generalizes this approach based on the spectral decomposition of a
max-stable process as described in \citet{dehaf06}. This will show in
particular that each generator may, without loss of generality, be
decomposed into a deterministic functional part and a univariate random one.
Furthermore this provides an important link towards functional extreme value
theory based on weak convergence as provided in \citet{dehaf06}.

\section{Generator processes for special $D$-norms} \label{sec:generators_special_D-norms}

In the following we construct a generator process $\bfZ$ that induces the
$D$-norm
\begin{equation}\label{eqn:D-norm_g_t_h_X}
  \dnorm*{\f}{\spectraldensalt} := \intdif{\R}{\supnorm{\f\spectraldensalt*{s}}}{s}, \quad \f\in\Espace,
\end{equation}
where $\spectraldensalt = \ocoll*{\spectralfunaltsubsup{t}{*}}{t}{\unit}$ is
a family of probability densities $\spectralfunaltsubsup{t}{*}$ on $\R$, \ie
$\spectraldensalt*{s} := \ocoll*{\spectralfunaltsubsup*{t}{*}{s}}{t}{\unit}$
is for every $s\in\R$ a non-negative function on $\unit$. A sufficient
condition for $\dnormd*{\spectraldensalt}$ actually being a $D$-norm is
given in \prettyref{prop:sufficient_condition_D-norm_g_t_h_X}.

\begin{lemma}\label{lem:generator_for_D-norm_g_t_h_X}
Let $\spectraldensalt$ be given as in \eqref{eqn:D-norm_g_t_h_X}. If
$X\in\R$ is a random variable with Lebesgue-density $\h>0$, then
\begin{equation}\label{eqn:generator_for_D-norm_g_t_h_X}
  Z^{\spectraldensalt,\h}_t := \frac{\spectralfunaltsubsup*{t}{*}{X}}{\h*{X}}, \qquad t\in\unit,
\end{equation}
defines a generator process $\bfZ_{\spectraldensalt,\h} =
\pro*[big]{Z^{\spectraldensalt,\h}_t}{t}{\unit}$ if and only if
\begin{equation}\label{eqn:sufficient_condition_D-norm_g_t_h_X}
  \spectraldensalt*{s}\text{ is continuous},\ s\in\R,
  \quad \text{and}\quad
  \intdif{\R}{\supnorm{\spectraldensalt*{s}}}{s} < \infty.
\end{equation}
\end{lemma}
\begin{proof}
Let $\spectraldensalt$ have the desired properties. Then we have obviously
$\bfZ_{\spectraldensalt,\h}\in\Cbarplus$,
\begin{equation*}
  \E{Z^{\spectraldensalt,\h}_t}
  = \intdif{\R}{\frac{\spectralfunaltsubsup*{t}{*}{s}}{\h*{s}} \h*{s}}{s}
  = 1, \quad t\in\unit,
\end{equation*}
and, analogously,
\begin{equation*}
  \E{\supnorm{\bfZ_{\spectraldensalt,\h}}}
  = \intdif{\R}{\supnorm{\spectraldensalt*{s}}}{s}
  < \infty.
\end{equation*}

If, on the other hand, $\bfZ_{\spectraldensalt,\h}$ is a generator process,
then we conclude $\intdif{\R}{\supnorm{\spectraldensalt*{s}}}{s} < \infty$
from $\E{\supnorm{\bfZ_{\spectraldensalt,\h}}} < \infty$. Since
$\E[big]{Z^{\spectraldensalt,\h}_t}$ and
$\E{\supnorm{\bfZ_{\spectraldensalt,\h}}}$ do depend on the distribution of
$X$ only, $\widetilde{\bfZ}_{\spectraldensalt,\h} =
\pro*[big]{\widetilde{Z}^{\spectraldensalt,\h}_t}{t}{\unit}$ defined by
\begin{equation*}
  \widetilde{Z}^{\spectraldensalt,\h}_t := \frac{\spectralfunaltsubsup*[big]{t}{*}{\widetilde{X}}}{\h*[big]{\widetilde{X}}}, \qquad t\in\unit,
\end{equation*}
is a generator process, too, where $\widetilde{X} = \mathrm{id}_\R$ is the
identity function on $\R$ equipped with its Borel-$\sigma$-algebra and the
probability measure induced by $\h$. The remaining assertion is, thus,
implied by the continuity of $\widetilde{\bfZ}_{\spectraldensalt,\h}$.
\end{proof}

\begin{proposition}\label{prop:sufficient_condition_D-norm_g_t_h_X}
If condition \eqref{eqn:sufficient_condition_D-norm_g_t_h_X} holds, then
$\dnormd*{\spectraldensalt}$ is a $D$-norm that is generated by
$\bfZ_{\spectraldensalt,\h}$.
\end{proposition}
\begin{proof}
Applying \prettyref{lem:generator_for_D-norm_g_t_h_X}, one obtains 
\begin{equation*}
  \begin{split}
    \E{\supnorm{\f\bfZ_{\spectraldensalt,\h}}}
    &= \intdif{\R}{\supnorm{\frac{\f\spectraldensalt*{s}}{\h*{s}}} \h*{s}}{s}
    = \dnorm*{\f}{\spectraldensalt},\quad \f\in \Espace,
  \end{split}
\end{equation*}
\ie $\dnormd*{\spectraldensalt}$ is actually a $D$-norm, see \citet{aulfaho11}.
\end{proof}

Note that the $D$-norm and, thus, the \emph{generator constant} $m =
\dnorm*{1}{\spectraldensalt} =
\intdif{\R}{\supnorm{\spectraldensalt*{s}}}{s}$ of
$\bfZ_{\spectraldensalt,\h}$ do \emph{not} depend on the choice of $\h$.

\begin{example}[\citet{aulfa11}]
Let $\psi : \R\to\intoo{0}{\infty}$ be a continuous probability density
having the properties $\operatorname{\psi}\paren{-s} =
\operatorname{\psi}\paren{s}$, $s\in\R$, and
$\operatorname{\psi}\paren{s_1}\ge \operatorname{\psi}\paren{s_2}$, $0\le
s_1< s_2$. For $\beta> 0$ we define $\operatorname{\psi_\beta}\paren{s} :=
\beta\operatorname{\psi}\paren{\beta s}$, $s\in\R$. Then the $D$-norm
\begin{equation*}
  \dnorm*{\f}{\psi_\beta} := \intdif{\R}{\sup_{t\in\unit}\!\paren{\abs{\f*{t}}\, \operatorname{\psi_\beta}\paren{s-t}}}{s}, \quad \f\in \Espace,
\end{equation*}
is generated by
\begin{equation*}
  \pro*{\frac{\psi_\beta\paren{X-t}}{\h*{X}}}{t}{\unit}
\end{equation*}
for every random variable $X$ with Lebesgue-density $\h>0$.
\end{example}

In the preceding example one may choose $\h = \operatorname{\psi_\beta}$ and
$\operatorname{\psi} = \operatorname{\varphi}$ where
$\operatorname{\varphi}$ denotes the density of the standard normal
distribution. Defining $\sigma := \beta^{-1}$ yields that the corresponding
$D$-norm $\dnormd*{\operatorname{\varphi_{0,\sigma^2}}}$ has by
\prettyref{prop:sufficient_condition_D-norm_g_t_h_X} the generator
\begin{equation*}
  \paren{\exp\paren{\frac{t\paren{2X-t}}{2\sigma^2}}}_{t\in\unit}
\end{equation*}
where $X\sim N(0,\sigma^2)$.

\section{A spectral decomposition of the generator} \label{sec:generator_decomposition}

Now we generalize the approach of constructing a generator process presented in
\prettyref{sec:generators_special_D-norms}. It turns out that one may assume without loss
of generality that a generator $\bfZ$ is the composition of a deterministic function $g :
\unit^2 \to \R$ and a random variable $U$ that is uniformly distributed on $\unit$, see
\prettyref{prop:generator_decomposition}. This reasoning shows that the construction of a
\emph{random} function $\bfZ$ that generates a $D$-norm is reduced to the problem of
finding a \emph{deterministic} function $g$ with suited properties as follows.
\begin{definition} \label{defn:generator_function}
If $\g : \unit^2\to\intco{0}{\infty}$ is a function satisfying
\begin{enumerate}
  \item \label{item:generator_function_C} $\g*{s,{\cdot}}$ is continuous, $\quad
      s\in\unit$,
  \item \label{item:generator_function_int_1} $\g*{{\cdot},t}$ is a Lebesgue
      probability density, $\quad t\in\unit$,
  \item \label{item:generator_function_int_sup_finite}
      $\intdif[1]{0}{\sup_{t\in\unit} \g*{s, t}}{s} < \infty$,
\end{enumerate}
then we call $\spectraldens := \ocoll*{\g*{{\cdot},t}}{t}{\unit}$ a family of
\emph{spectral functions}, according to \citet[Remark\,9.6.2]{dehaf06}, or,
more precisely, a family of \emph{spectral densities}. In this case we
define $\spectraldens*{s} := \ocoll{\g*{s,t}}{t}{\unit} =\g*{s,{\cdot}}$.
\end{definition}

Let $U$ be a random variable that is uniformly distributed on $\unit$. Then
$\spectraldens*{U}$ is obviously a generator process and, thus, gives rise
to a uniquely determined $D$-norm. On the other hand, for each $D$-norm
$\dnormd$ there is a family of spectral densities $\spectraldens$ such that
$\spectraldens*{U}$ is a generator of $\dnormd$, \ie
\begin{equation*}
  \dnorm{\f} = \E{\supnorm{\f\spectraldens*{U}}},
\end{equation*}
giving the desired decomposition of a generator into a deterministic
functional part and a \emph{univariate} random one. The latter assertion is
a consequence of the following result.
It is implied by \citet[Theorem\,9.6.1]{dehaf06} which itself is a
conclusion of \citet{resroy91} and \citet{dehaan84}. For convenience we also
state its proof.

\begin{proposition}[\citet{dehaf06}] \label{prop:generator_decomposition}
Let $\bfeta = \opro{\eta}{t}{\unit}\in \Runit$ be a stochastic process. Then
$\bfeta$ is a standard MSP in $\Cspace$ if and only if there is a family of
spectral densities $\spectraldens$ such that
\begin{equation} \label{eqn:generator_decomposition}
  \P{\bfeta\le \f} = \exp\paren[big]{- \E{\supnorm{\f\spectraldens*{U}}}}, \qquad \f\in\Ebarmin,\ U\sim\Ud.
\end{equation}
\end{proposition}
\begin{proof}
Denote by $N$ a Poisson point process with points
$\paren{R_i,S_i}\in\intoc{0}{\infty}\times\unit$, $i\in\N$, and intensity
measure $\nu$ defined by $\nu\paren{B} =
\intdif{\unit}{\!\intdif{\intoc{0}{\infty}}{\ind B(r,s)\, r^{-2}}{r}}{s}$
where $B$ is a Borel set in $\intoc{0}{\infty}\times\unit$.
\citet[Theorem\,2.1.1]{reiss93} assures that $N$ actually exists.

Let $\bfeta\in \Cspace$ be a standard MSP. Note that \cite{ginhv90} as well
as \citet[Lemma\,1]{aulfaho11} imply $P\paren[big]{\sup_{t\in\unit} \eta_t<0
} = 1$ and thus $\ocoll*{-1\big/\eta_t}{t}{\unit}$ is simple max-stable. Now
we have by \citet[Theorem\,9.6.1]{dehaf06} that there is a family of
spectral densities $\spectraldens = \coll*{\g*{{\cdot},t}}{t}{\unit}$ satisfying
\begin{align*}
  \P{\bfeta\le \f}
  &= \P{\sup_{i\in\N}\paren{R_i \g*{S_i,t}} \le \abs{\f*{t}}^{-1},\ t\in\unit} \\
  &= \P{N\paren{\set*{\paren{r,s}\in \intoc{0}{\infty}\times\unit}{\exists_{t\in\unit}\, r \g*{s,t} > \abs{\f*{t}}^{-1}}}=0} \\
  &= \exp*{-\nu\paren{\set*{\paren{r,s}\in \intoc{0}{\infty}\times\unit}{r > \brac{\sup_{t\in\unit}\!\paren{\abs{\f*{t}}\g*{s,t}}}^{-1} }}} \\
  &= \exp*{- \intdif{\unit}{\!\intdif{\intoc{\supnorme{\f\spectraldens*{s}}{-1}}{\infty}}{\frac{1}{r^2}}{r}}{s} } \\
  &= \exp*{- \intdif{\unit}{\supnorm{\f\spectraldens*{s}}}{s} } \\
  &= \exp*[big]{- \E{\supnorm{\f\spectraldens*{U}}} }, \qquad \f\in\Ebarmin,
\end{align*}
where $U$ is uniformly distributed on $\unit$.

Now let \eqref{eqn:generator_decomposition} hold for $\bfeta\in\Runit$ and
some family of spectral densities $\spectraldens =
\coll*{\g*{{\cdot},t}}{t}{\unit}$. Since
\begin{equation*}
    \P{\sup_{t\in\unit}\!\eta_t < 0}
    = \P{\bigcup_{n\in\N} \set{\eta_t\le -1/n,\, t\in\unit}}
    = \lim_{n\in\N} \P{\bfeta\le -1/n}
    = 1,
\end{equation*}
we obtain
\begin{equation*}
  \P{-\frac{1}{\bfeta}\le \f}
  = \P{\sup_{i\in\N}\paren{R_i \g*{S_i,t}} \le \f*{t},\ t\in\unit}, \qquad \f\in \Eplus,
\end{equation*}
repeating the arguments from above. This gives
\begin{equation*}
  -\frac{1}{\bfeta} \stackrel{d}{=} \coll*{\sup_{i\in\N}\paren{R_i \g*{S_i,t}}}{t}{\unit},
\end{equation*}
\ie $\bfeta$ is by \citet[Theorem\,9.6.1]{dehaf06} a standard MSP in
$\Cspace$; see also \citet[Theorem\,3.2]{resroy91}.
\end{proof}

We summarize our previous results:
\begin{proposition} \label{prop:D-norm_g}
For each $D$-norm $\dnormd$ there is a family of spectral densities
$\spectraldens = \coll*{\g*{{\cdot},t}}{t}{\unit}$ such that
\begin{equation} \label{eqn:D-norm_g}
  \dnorm{\f} = \intdif{\intcc{0}{1}}{\sup_{t\in\unit}\!\paren{\abs{\f*{t}} \g*{s,t}}}{s},\qquad \f\in \Espace,
\end{equation}
and $\dnormd$ is, thus, generated by $\spectraldens*{U}$ where $U\sim\Ud$.

Conversely, if $\spectraldens$ is a given family of spectral densities, then
\eqref{eqn:D-norm_g} defines a $D$-norm.
\end{proposition}
\begin{proof}
While the first assertion follows directly from
\prettyref{prop:generator_decomposition}, the second one is implied by
\citet[Proposition\,2.3]{aulfaho11}; we also refer to
\citet[Proposition\,3.2]{ginhv90} and \citet[Theorem\,9.4.1]{dehaf06}.
\end{proof}

\begin{example} Let $U$ be uniformly distributed on $\unit$. 
\begin{enumerate}
  \item The supremum norm $\supnormd$ is a $D$-norm generated by
      $\spectraldens*{U} \equiv 2U$.
  \item Let $\spectraldensalt =
      \coll*{\spectralfunaltsubsup{t}{*}}{t}{\unit}$ and
      $\dnormd*{\spectraldensalt}$ be given as in
      \eqref{eqn:D-norm_g_t_h_X} and
      \eqref{eqn:sufficient_condition_D-norm_g_t_h_X}. If $\h>0$ is an
      arbitrarily chosen Lebesgue probability density on $\R$, then we
      obtain
      \begin{equation*}
        \begin{split}
          \dnorm*{\f}{\spectraldensalt}
          &= \intdif*{\R}{\sup_{t\in\unit}\!\paren{\abs{\f*{t}}\, \frac{\gsubsup*{t}{*}{s}}{\h*{s}}}}{\paren{P\ast H^{-1}\paren{U}}}{s} \\
          &= \intdif{\intoo{0}{1}}{\sup_{t\in\unit}\!\paren{\abs{\f*{t}}\, \frac{\gsubsup*{t}{*}{H^{-1}\paren{s}}}{\h*{H^{-1}\paren{s}}}}}{s}, \qquad \f\in \Espace,
        \end{split}
      \end{equation*}
      where $H(x) := \intdif[x]{-\infty}{\h*{s}}{s}$, $x\in\R$, and
      $H^{-1}(u) := \inf\set*{x\in\R}{H(x)\ge u}$, $u\in\unit*$.
      Furthermore, the function $\gsub{\h} : \unit^2\to\intco{0}{\infty}$
      given by
      \begin{equation*}
        \gsub*{\h}{s,t} := \frac{\gsubsup*[big]{t}{*}{H^{-1}\paren{s}}}{\h*[big]{H^{-1}\paren{s}}}, \ s\in\intoo{0}{1},\quad \text{and}\quad \gsub*{\h}{s,t} := 0, \ s\in\set{0,1},
      \end{equation*}
      defines a family of spectral densities $\spectraldenssub{\h}$ and
      thus
      \begin{equation*}
        \coll*{\frac{\gsubsup*[big]{t}{*}{H^{-1}\paren{U}}}{\h*[big]{H^{-1}\paren{U}}}}{t}{\unit}
      \end{equation*}
      is a generator of $\dnormd*{\spectraldensalt}$. This is in
      accordance with \prettyref{prop:sufficient_condition_D-norm_g_t_h_X}
      as the same generator process is obtained replacing $X$ in
      \eqref{eqn:generator_for_D-norm_g_t_h_X} with $H^{-1}\paren{U}$.
\end{enumerate}
\end{example}

\end{document}